\documentclass[a4paper, 11pt]{article}
\title{Complexe de poids des vari\'et\'es alg\'ebriques r\'eelles avec action}
\author{Fabien Priziac}
\date{}
\usepackage[french]{babel}
\usepackage{amsfonts}
\usepackage{amsmath}
\usepackage{amsthm}
\usepackage{amssymb}
\usepackage[all]{xy}
\usepackage{graphicx}

\newtheorem{de}{D\'efinition}[section]
\newtheorem{theo}[de]{Th\'eor\`eme}
\newtheorem{prop}[de]{Proposition}

\newtheorem{lem}[de]{Lemme}

\newcommand{\Supp}{\text{Supp}}
\newcommand{\modulo}{\text{mod}}

\theoremstyle{remark}
\newtheorem{rem}[de]{Remarque}

\begin{document}

\maketitle

\begin{abstract}
En utilisant la fonctorialit\'e du complexe de poids de C. McCrory et A. Parusi\'nski -qui induit un analogue de la filtration par le poids pour les vari\'et\'es alg\'ebriques complexes sur l'homologie de Borel-Moore \`a coefficients dans $\mathbb{Z}_2$ des vari\'et\'es alg\'ebriques r\'eelles-, on d\'efinit un complexe de poids avec action sur les vari\'et\'es alg\'ebriques r\'eelles munies d'une action d'un groupe fini. Mettant l'accent sur le groupe \`a deux \'el\'ements, on \'etablit ensuite une version filtr\'ee de la suite courte de Smith pour une involution, tenant compte de la filtration Nash-constructible qui r\'ealise le complexe de poids avec action. Son exactitude est impliqu\'ee par le d\'ecoupage d'une vari\'et\'e Nash munie d'une involution alg\'ebrique le long d'un sous-ensemble sym\'etrique par arcs.
\end{abstract}
\footnote{Mots-cl\'es : filtration par le poids, vari\'et\'es alg\'ebriques r\'eelles, action de groupe, suite exacte de Smith, ensembles sym\'etriques par arcs, fonctions Nash-constructibles, vari\'et\'es Nash.

Travail issu d'une th\`ese effectu\'ee au laboratoire IRMAR (Institut de Recherche Math\'ematique de Rennes), au sein de l'Universit\'e de Rennes 1, et partiellement financ\'e par le projet ANR-08-JCJC-0118-01.}

\section{Introduction}

Dans \cite{Del}, P. Deligne a \'etabli l'existence d'une filtration dite par le poids sur la cohomologie rationnelle des vari\'et\'es alg\'ebriques complexes. Un analogue de cette filtration sur l'homologie de Borel-Moore \`a coefficients dans $\mathbb{Z}_2$ des vari\'et\'es alg\'ebriques r\'eelles a \'et\'e introduit par B. Totaro dans \cite{Tot}. En utilisant un crit\`ere d'extension de foncteurs d\'efinis sur les vari\'et\'es lisses de F. Guill\'en et V. Navarro Aznar (\cite{GNA}), C. McCrory et A. Parusi\'nski ont montr\'e dans \cite{MCP} l'existence d'un complexe de cha\^ines filtr\'e nomm\'e de poids, d\'efini \`a quasi-isomorphisme filtr\'e pr\`es, fonctoriel sur la cat\'egorie des vari\'et\'es alg\'ebriques r\'eelles et induisant la filtration par le poids de Totaro. La suite spectrale de poids r\'eelle ne converge pas d\`es l'ordre deux (contrairement \`a son homologue complexe), et McCrory et Parusi\'nski mettent en \'evidence la richesse des informations qu'elle contient -par exemple les nombres de Betti virtuels (\cite{MCP-VB})-, tout en en enrichissant la compr\'ehension, en r\'ealisant par exemple le complexe de poids d\`es le niveau des cha\^ines par une filtration g\'eom\'etrique. Leur filtration Nash-constructible -d\'efinie en utilisant les fonctions Nash-constructibles- l'\'etend \`a la cat\'egorie des ensembles et morphismes $\mathcal{AS}$ (\cite{Kur}, \cite{KP}), \'etablissant la fonctorialit\'e des objets ``de poids'' par rapport aux applications continues avec graphe $\mathcal{AS}$.

Dans cet article, on consid\`ere des vari\'et\'es alg\'ebriques r\'eelles munies de l'action d'un groupe fini. On met alors \`a profit la fonctorialit\'e du complexe de poids pour le munir de l'action induite (th\'eor\`eme \ref{th_complex-poids-act}). D'une part, la finitude du groupe, impliquant l'existence d'une r\'esolution des singularit\'es \'equivariante, d'une compactification \'equivariante et d'un lemme de Chow-Hironaka \'equivariant (\cite{DL}, remarque \ref{val_crit_action}), nous permet de montrer, en utilisant une version avec action du crit\`ere d'extension de Guill\'en et Navarro Aznar (th\'eor\`eme \ref{crit_action}), l'unicit\'e (\`a quasi-isomorphisme filtr\'e \'equivariant pr\`es) du complexe de poids avec action. D'autre part, l'\'equivariance des op\'erations de base sur les cha\^ines semi-alg\'ebriques \`a supports ferm\'es nous permet de le r\'ealiser via la filtration g\'eom\'etrique/Nash-constructible, munie de l'action induite \'egalement par fonctorialit\'e (remarque \ref{rem_fil_geom_act}).

Dans la suite de l'article, on souhaite comprendre plus en avant les cha\^ines invariantes sous l'action, et notamment la compatibilit\'e de la suite exacte courte de Smith avec la filtration Nash-constructible. Pour cela, on se concentre sur le cas du groupe \`a deux \'el\'ements, pour lequel cette \'etude se r\'ev\`ele d\'ej\`a tr\`es int\'eressante. On montre ainsi que l'on peut avoir un certain contr\^ole sur la r\'egularit\'e du d\'ecoupage d'une cha\^ine invariante propos\'e par la suite exacte de Smith d'une involution. Pr\'ecis\'ement, on \'etablit que, pour tout entier $\alpha$, on peut \'ecrire une cha\^ine invariante de degr\'e $\alpha$ (vis-\`a-vis de la filtration Nash-constructible) comme la somme de sa restriction \`a l'ensemble des points fixes, d'une cha\^ine de degr\'e au plus $\alpha + 1$ et de l'image de celle-ci par l'involution (proposition \ref{decoup_fin}). 

Pour parvenir \`a ce r\'esultat, on r\'eduit la question au cas d'une vari\'et\'e Nash affine compacte connexe munie d'une involution alg\'ebrique (non triviale). On prouve qu'un tel objet peut \^etre d\'ecoup\'e le long d'un sous-ensemble sym\'etrique par arcs en deux morceaux semi-alg\'ebriques \'echang\'es sous l'action (th\'eor\`eme \ref{decoup_Nash}). 

Enfin, en guise d'interpr\'etation de cette suite exacte courte de Smith Nash-constructible, on montre que dans le cas d'une vari\'et\'e alg\'ebrique compacte munie d'une action libre, l'isomorphisme entre les cha\^ines invariantes et les cha\^ines du quotient, qui est alors sym\'etrique par arcs, est compatible avec la filtration Nash-constructible (proposition \ref{fil_Nash_quotient}).
\\

On d\'ebute cet article en munissant, dans la partie \ref{chaines semi-alg}, le complexe des cha\^ines semi-alg\'ebriques \`a supports ferm\'es d'une vari\'et\'e alg\'ebrique r\'eelle munie d'une action de groupe, de l'action induite. On remarque alors l'\'equivariance de la fonctorialit\'e des cha\^ines semi-alg\'ebriques, ainsi que celle des op\'erations de restriction, d'adh\'erence et de tir\'e en arri\`ere.

C'est dans la partie \ref{chap_comp_poids_act} que le complexe de poids des vari\'et\'es alg\'ebriques r\'eelles avec action d'un groupe fini est muni de l'action induite par fonctorialit\'e. Y est \'egalement \'etablie son unicit\'e \`a quasi-isomorphisme filtr\'e \'equivariant pr\`es, en utilisant l'analogue avec action du crit\`ere d'extension de Guill\'en et Navarro Aznar.

L'existence d'un d\'ecoupage d'une vari\'et\'e Nash munie d'une involution alg\'ebrique le long d'un sous-ensemble sym\'etrique par arcs est \'enonc\'ee et d\'emontr\'ee dans la partie \ref{section_decoup_var_nash}. On l'utilise ensuite dans la partie \ref{suite_exacte_smith_nash_cons} pour \'etablir notre suite exacte courte de Smith Nash-constructible, pour laquelle on y donne aussi notre interp\'etation dans le cas compact sans point fixe.
\\

{\bf Remerciements.} L'auteur souhaite remercier M. Coste, G. Fichou, F. Guill\'en, T. Limoges, C. McCrory et A. Parusi\'nski pour de fructueux discussions et commentaires.

\section{Cha\^ines semi-alg\'ebriques et action de groupe} \label{chaines semi-alg}

Consid\'erant une action sur un ensemble semi-alg\'ebrique localement compact, on munit le complexe de ses cha\^ines semi-alg\'ebriques \`a supports ferm\'es (\cite{MCP} Appendix) de l'action induite par fonctoralit\'e. On v\'erifie alors que l'action commute avec les op\'erations de restriction, d'adh\'erence (proposition \ref{prop_res,adh_action}) et de tir\'e en arri\`ere (proposition \ref{prop_pullback_action}).
\\

Dans cet article, suivant la d\'efinition de \cite{MCP}, une vari\'et\'e alg\'ebrique r\'eelle d\'esignera un sch\'ema r\'eduit s\'epar\'e de type fini sur $\mathbb{R}$.

Soit donc $X$ un sous-ensemble semi-alg\'ebrique de l'ensemble des points r\'eels d'une vari\'et\'e alg\'ebrique r\'eelle. Pour la commodit\'e du lecteur, on rappelle la d\'efinition du complexe des cha\^ines semi-alg\'ebriques \`a supports ferm\'es de $X$ (\cite{MCP} Appendix) :

\begin{de} Pour tout $k \geq 0$, on note $C_k(X)$ le quotient du $\mathbb{Z}_2$-espace vectoriel engendr\'e par les sous-ensembles semi-alg\'ebriques ferm\'es de $X$ de dimension $\leq k$, par les relations
\begin{itemize}
	\item la somme $A + B$ est \'equivalente \`a l'adh\'erence dans $X$ (pour la topologie forte) $cl_X(A \div B)$ de la diff\'erence sym\'etrique de $A$ et de $B$,
	\item la classe de $A$ est nulle si la dimension de $A$ est strictement plus petite que $k$.
\end{itemize}
Une cha\^ine semi-alg\'ebrique \`a supports ferm\'es de dimension $k$ de $X$ est une classe d'\'equivalence de $C_k(X)$. Toute cha\^ine $c$ de $C_k(X)$ peut s'\'ecrire comme la classe d'un sous-ensemble semi-alg\'ebrique ferm\'e $A$ de $X$ de dimension $\leq k$, not\'ee $[A]$.

L'op\'erateur de bord $\partial_k : C_k(X) \rightarrow C_{k-1}(X)$ est d\'efini par, si $c = [A] \in C_k(X)$, $\partial_k c = [\partial A]$ o\`u $\partial A$ d\'esigne le bord semi-alg\'ebrique $\{x \in A~|~\Lambda \mathbf{1}_A(X) \equiv 1~{\rm mod}~2\}$ de $A$ ($\Lambda$ est l'entrelacs sur les fonctions constructibles : cf \cite{MCP-ACF}).
\end{de}

On renvoie \`a l'appendice de \cite{MCP} pour les d\'efinitions du pouss\'e en avant, des op\'erations de restriction et d'adh\'erence, ainsi que celle du tir\'e en arri\`ere.
\\

On suppose maintenant que $X$ est muni d'une action d'un groupe $G$ qui agit par hom\'eo-morphismes semi-alg\'ebriques : pour tout $g \in G$, on note $\alpha_g : X \rightarrow X$ l'hom\'eomorphisme semi-alg\'ebrique associ\'e  \`a $g$.

Par fonctorialit\'e, on obtient alors une action sur le complexe $C_*(X)$ par isomorphismes lin\'eaires donn\'ee par ${\alpha_{g}}_* : C_*(X) \rightarrow C_*(X)$ pour $g \in G$. 

Le complexe $C_*(X)$ muni de cette action de $G$ devient alors un $G$-complexe.
\\

\begin{lem} Soient $Y$ un autre ensemble semi-alg\'ebrique localement compact muni d'une action de $G$ et $f : X \rightarrow Y$ une application semi-alg\'ebrique continue propre et \'equivariante. Alors le pouss\'e en avant $f_*$ est \'equivariant par rapport aux actions de $G$ sur les $G$-complexes $C_*(X)$ et $C_*(Y)$.
\end{lem}

\begin{proof} Fonctorialit\'e des cha\^ines semi-alg\'ebriques \`a supports ferm\'es (\cite{MCP} Appendix).
\end{proof}

L'action de $G$ sur une cha\^ine semi-alg\'ebrique \`a support ferm\'e peut \^etre donn\'ee par l'action sur l'un de ses repr\'esentants :

\begin{prop}
Soient $g \in G$ et $c = [A] \in C_k(X)$, alors
$$g.c = [g.A].$$
\end{prop} 

\begin{proof} On a $g.c := {\alpha_{g}}_* (c) = [B]$, o\`u
$$B = cl \{x \in X ~|~ {\alpha_{g}}_* \mathbf{1}_A(x) =  \chi(\alpha_g^{-1}(x) \cap A) \equiv 1 \mbox{ mod $2$ } \}.$$
Or pour $x \in X$, $\chi(\alpha^{-1}_g(x) \cap A) = 1$ si $\alpha_g^{-1}(x) \in A$ i.e. $x \in \alpha_g(A)$, $0$ si $\alpha_g^{-1}(x) \notin A$ i.e. $x \notin \alpha_g(A)$ (car $\alpha_g$ est un hom\'eomorphisme). Donc $B = cl \left(\alpha_g(A)\right) = \alpha_g(A)$.
\end{proof}

A partir de cette constatation, on peut montrer simplement la commutativit\'e de l'action avec les op\'erations de restriction, d'adh\'erence et le tir\'e en arri\`ere :

\begin{prop} \label{prop_res,adh_action}
Soit $Z \subset X$ un sous-ensemble semi-alg\'ebrique localement ferm\'e globalement stable sous l'action de $G$. La restriction \`a $Z$ des cha\^ines semi-alg\'ebriques \`a supports ferm\'es de $X$ commute avec l'action induite de $G$. De m\^eme, l'op\'eration d'adh\'erence des cha\^ines semi-alg\'ebriques \`a support ferm\'es de $Z$ commute avec l'action induite de $G$.
\end{prop}

\begin{proof} Soient $k \geq 0$ et $c = [A] \in C_k(X)$. Alors,
$$g.(c |_Z) = [\alpha_g(A \cap Z)] = [\alpha_g(A) \cap Z] = (g.c) |_Z.$$
 
Soient $k \geq 0$ et $c = [A] \in C_k(Z)$. Alors,
$$\overline{g.c} = [cl(\alpha_g(A)] = [\alpha_g(cl(A))] = g.\bar{c}.$$
\end{proof}

\begin{prop} \label{prop_pullback_action}
On consid\`ere le diagramme commutatif d'ensembles semi-alg\'ebriques localement ferm\'es suivant :
$$\begin{array}{ccl}
\ \widetilde{Y} & \rightarrow & \widetilde{X} \\
  \downarrow &      &  \downarrow \pi \\
  Y & \xrightarrow{i} & X 
  \end{array}$$
  o\`u $\pi : \widetilde{X} \rightarrow X$ est une application semi-alg\'ebrique continue propre, $i$ est l'inclusion d'un sous-ensemble semi-alg\'ebrique ferm\'e, $\widetilde{Y} = \pi^{-1} (Y)$, et la restriction $\pi : \widetilde{X} \setminus \widetilde{Y} \rightarrow X \setminus Y$ est un hom\'eomorphisme, et o\`u tous les objets sont munis d'une action de $G$ et tous les morphismes sont \'equivariants par rapport \`a ces actions.
 
L'action de $G$ sur les cha\^ines semi-alg\'ebriques \`a supports ferm\'es commute alors avec le tir\'e en arri\`ere.
\end{prop}

\begin{proof}L'action de $G$ commute avec l'adh\'erence, le pouss\'e en avant, la restriction donc avec le tir\'e en arri\`ere.
\end{proof}

On \'etablit enfin deux autres lemmes, utilis\'es dans la preuve du corollaire \ref{decoup_fin}. Dans ce r\'esultat fondamental de notre \'etude, on parviendra \`a d\'ecouper une cha\^ine semi-alg\'ebrique invariante sous une involution, modulo sa restriction aux points fixes, en deux cha\^ines images l'une de l'autre et dont le ``degr\'e de r\'egularit\'e'' (par rapport \`a la filtration dite Nash-constructible de \cite{MCP}) ne s'\'eloignera pas trop du degr\'e initial de la cha\^ine de d\'epart.

\begin{lem} \label{lem_1_2_3}Soient $A_1$, $A_2$ et $A_3$ trois sous-ensembles semi-alg\'ebriques ferm\'es de $X$ de dimension $k$. Si $[A_1] = [A_2]$ dans $C_k(X)$, alors
$[A_1 \cap A_3] = [A_2 \cap A_3]$.
\end{lem}

\begin{proof} On a $[A_1 \cap A_3] + [A_2 \cap A_3] = [cl\left( \left(A_1 \cap A_3\right) \div \left(A_2 \cap A_3\right)\right)]$, or $\left(A_1 \cap A_3\right) \div \left(A_2 \cap A_3\right)= (A_1 \div A_2) \cap A_3$ et $\dim ( A_1 \div A_2) < k$ car $[A_1] = [A_2]$.
\end{proof}

\begin{lem} \label{lem_1_2_cl} Soient $A_1$, $A_2$ deux sous-ensembles semi-alg\'ebriques de $X$ de dimension $k$ avec $A_1$ ferm\'e dans $X$. Alors, dans $C_k(X)$, $[cl(A_1 \cap A_2)] = [A_1 \cap cl(A_2)]$.
\end{lem}

\begin{proof} On a 
$$cl(A_1 \cap A_2) \div \left(A_1 \cap cl(A_2) \right) = A_1 \cap cl(A_2) \setminus cl(A_1 \cap A_2) \subset A_1 \cap cl(A_2) \setminus A_1 \cap A_2 = A_1 \cap (cl(A_2) \setminus A_2)$$
et $\dim( cl(A_2)\setminus A_2) < k$.
\end{proof}

\section{Complexe de poids avec action} \label{chap_comp_poids_act}

Soit $G$ un groupe fini.  

On d\'efinit (th\'eor\`eme \ref{th_complex-poids-act}) un complexe de poids avec action de $G$ sur la cat\'egorie de ce que l'on appellera les $G$-vari\'et\'es alg\'ebriques r\'eelles, fonctoriel par rapport aux morphismes propres r\'eguliers \'equivariants, et unique \`a quasi-isomorphisme filtr\'e \'equivariant pr\`es avec les propri\'et\'es d'extension, d'acyclicit\'e et d'additivit\'e analogues \`a celles du cadre sans action de \cite{MCP} Theorem 1.1. Pour son existence, cela consistera \`a munir le complexe de poids de McCrory-Parusi\'nski de l'action de $G$ induite par fonctorialit\'e. Quant \`a son unicit\'e, elle sera donn\'ee par une version avec action du crit\`ere d'extension de F. Guill\'en et V. Navarro Aznar (\cite{GNA} Th\'eor\`eme 2.2.2), qui justifie la restriction au cas d'un groupe fini pour lequel il existe une compactification \'equivariante, un lemme de Chow-Hironaka \'equivariant ainsi qu'une r\'esolution des singularit\'es \'equivariante.

On pr\^etera ensuite attention au groupe $G = \mathbb{Z}/2\mathbb{Z}$. En effet, d\`es le plus petit groupe non trivial, les structures avec action s'enrichissent consid\'erablement. Dans la section \ref{suite_exacte_smith_nash_cons}, on \'etablit une version de la suite exacte courte de Smith qui tient compte de la filtration Nash-constructible (qui r\'ealise le complexe de poids par \cite{MCP} Theorem 2.8 et Corollary 3.11). Ce r\'esultat repose sur un th\'eor\`eme de d\'ecoupage des vari\'et\'es Nash munies d'une involution alg\'ebrique, que l'on \'etablit dans la section \ref{section_decoup_var_nash}.
\\

Tout d'abord, d\'efinissons pr\'ecis\'ement les cat\'egories sur lesquelles nous allons travailler. Dans toute la suite, une action de $G$ par isomorphismes alg\'ebriques sur une vari\'et\'e alg\'ebrique r\'eelle $X$ d\'esignera une action par isomorphismes de sch\'emas telle que l'orbite de tout point de $X$ est contenue dans un sous-sch\'ema ouvert affine.

\begin{de} On note 
\begin{itemize}
	\item $\mathbf{Sch}_c^G(\mathbb{R})$ la cat\'egorie des vari\'et\'es alg\'ebriques r\'eelles munies d'une action de $G$ par isomorphismes alg\'ebriques -on nomme de tels objets des $G$-vari\'et\'es alg\'ebriques r\'eelles- et des morphismes propres r\'eguliers \'equivariants,
	\item $\mathbf{Reg}_{comp}^G(\mathbb{R})$ la sous-cat\'egorie des $G$-vari\'et\'es compactes non singuli\`eres,
	\item $\mathbf{V}^G(\mathbb{R})$ la sous-cat\'egorie des $G$-vari\'et\'es projectives non singuli\`eres.
\end{itemize}
On note \'egalement 
\begin{itemize}
	\item $\mathcal{C}^G$ la cat\'egorie des $G$-complexes de $\mathbb{Z}_2$-espaces vectoriels born\'es munis d'une filtration croissante born\'ee par des $G$-complexes avec inclusions \'equivariantes -on nomme de tels objets des $G$-complexes filtr\'es- et des morphismes de complexes filtr\'es \'equivariants,
	\item $\mathcal{D}^G$ la cat\'egorie des $G$-complexes born\'es et des morphismes de complexes \'equivariants.
\end{itemize}
\end{de}

\begin{rem} 

\begin{itemize} 
	\item Si $X$ est une $G$-vari\'et\'e alg\'ebrique r\'eelle, on a vu que l'action de $G$ sur~$X$ induit, par fonctorialit\'e de $C_* : \mathbf{Sch}_c(\mathbb{R}) \rightarrow \mathcal{D}~;~X \mapsto C_*(X) := C_*(X(\mathbb{R}))$, une action de $G$ sur $C_*(X)$, et donc \'egalement une action (lin\'eaire) sur l'homologie $H_*(X) := H_*(C_*(X))$, qui est l'homologie de Borel-Moore de l'ensemble des points r\'eels de $X$ \`a coefficients dans~$\mathbb{Z}_2$ (\cite{MCP} section 1). On a ainsi un foncteur
$$C_* : \mathbf{Sch}_c^G(\mathbb{R}) \rightarrow \mathcal{D}^G~;~X \mapsto C_*(X).$$
	\item Si $(K_*, F_*)$ est un $G$-complexe filtr\'e, la suite spectrale induite est naturellement munie d'une action de groupe : chaque terme est munie de l'action de $G$ induite, et les diff\'erentielles sont \'equivariantes pour celle-ci.
	\item Le complexe simple filtr\'e associ\'e \`a un diagramme cubique dans $\mathcal{C}^G$ (\cite{MCP} section 1.1) peut \^etre naturellement muni de l'action du groupe $G$ induite (en consid\'erant l'action diagonale sur les sommes directes), et devenir ainsi un \'el\'ement de $\mathcal{C}^G$.
\end{itemize}

\end{rem}

Dans la continuit\'e de ce qu'ont fait C. McCrory et A. Parusi\'nski dans le cadre sans action de groupe, on s'int\'eresse aux morphismes filtr\'es qui induisent des isomorphismes au niveau des suites spectrales :

\begin{de} On note $H o ~ \mathcal{C}^G$ la cat\'egorie $\mathcal{C}^G$ localis\'ee par rapport aux quasi-isomorphismes filtr\'es \'equivariants, autrement appel\'es quasi-isomorphismes de $\mathcal{C}^G$, i.e. les morphismes filtr\'es \'equivariants entre $G$-complexes filtr\'es qui induisent un isomorphisme (\'equivariant) au niveau $E^1$ des suites spectrales induites.
\end{de}

\begin{rem} Tout $G$-complexe $K_*$ de $\mathcal{D}^G$ peut \^etre muni de la filtration canonique (\cite{MCP} section 1.1) sur laquelle agit naturellement le groupe $G$ (le noyau de la diff\'erentielle de $K_*$, qui est \'equivariante, est stable sous l'action de $G$), et ainsi $(K_*,F^{can}_*)$ est un \'el\'ement de $\mathcal{C}^G$.

De cette fa\c{c}on, un quasi-isomorphisme \'equivariant entre $G$-complexes born\'es induit un quasi-isomorphisme filtr\'e \'equivariant entre $G$-complexes filtr\'es munis de la filtration canonique.
\end{rem}

Nous allons \`a pr\'esent pouvoir \'enoncer le th\'eor\`eme d'existence et d'unicit\'e (dans $H o ~ \mathcal{C}^G$) du complexe de poids avec action du groupe $G$.

L'existence sera donn\'ee directement par la fonctorialit\'e du complexe de poids (sans action) de C. McCrory et A. Parusi\'nski. Son unicit\'e requerra une version avec action du crit\`ere d'extension de F. Guill\'en et V. Navarro Aznar \'enonc\'e dans \cite{GNA} :

\begin{theo} \label{crit_action} Soient $\mathcal{C}$ une cat\'egorie de descente cohomologique et  
$$F : \mathbf{V}^G(\mathbb{R}) \longrightarrow H o ~ \mathcal{C}$$
un foncteur contravariant $\Phi$-rectifi\'e qui v\'erifie 
\begin{itemize}
\item[(F1)] $F(\emptyset) = 0$, et le morphisme canonique $F(X \sqcup Y) \rightarrow F(X) \times F(Y)$ est un isomorphisme (dans $H o ~ \mathcal{C}$),
\item[(F2)] si $X_{\bullet}$ est un carr\'e acyclique \'el\'ementaire de $\mathbf{V}^G(\mathbb{R})$, $\mathbf{s}F(X_{\bullet})$ est acyclique.
\end{itemize}
Alors, il existe une extension de $F$ en un foncteur contravariant $\Phi$-rectifi\'e
$$F_c : \mathbf{Sch}_c^G(\mathbb{R}) \rightarrow H o ~ \mathcal{C}$$
telle que :
\begin{enumerate}
	\item si $X_{\bullet}$ est un carr\'e acyclique de $\mathbf{Sch}_c^G(\mathbb{R})$, $\mathbf{s}F_c(X_{\bullet})$ est acyclique,
	\item si $Y$ est une sous-vari\'et\'e ferm\'ee de $X$ stable sous l'action de $G$ sur $X$, on a un isomorphisme naturel (dans $H o ~ \mathcal{C}$)
$$F_c(X \setminus Y) \cong \mathbf{s} (F_c(X) \rightarrow F_c(Y)).$$
\end{enumerate}
En outre, cette extension est unique, \`a isomorphisme unique pr\`es.
\end{theo}

\begin{rem} \label{val_crit_action} Le crit\`ere d'extension (\cite{GNA} Th\'eor\`eme 2.2.2) reste bien valable dans ce contexte. En effet, celui-ci n\'ecessite, sur la cat\'egorie de vari\'et\'es consid\'er\'ee, une r\'esolution des singularit\'es, un lemme de Chow-Hironaka et une compactification. Or, sur la cat\'egorie $\mathbf{Sch}_c^G(\mathbb{R})$, comme le groupe~$G$ est d'ordre fini, une r\'esolution des singularit\'es \'equivariante, un lemme de Chow-Hironaka \'equivariant ainsi qu'une compactification \'equivariante existent par \cite{DL} (Appendix). 
\end{rem}

\begin{theo} \label{th_complex-poids-act} Le foncteur 
$$F^{can} C_{*} : \mathbf{V}^G(\mathbb{R}) \longrightarrow H o ~ \mathcal{C}^G~;~X \mapsto F^{can} C_*(X)$$
admet une extension en un foncteur 
$${}^G\!\mathcal{W} C_{*} : \mathbf{Sch}_c^G(\mathbb{R}) \longrightarrow H o ~ \mathcal{C}^G$$
d\'efini pour toutes les $G$-vari\'et\'es alg\'ebriques r\'eelles et tous les morphismes propres r\'eguliers \'equivariants, qui v\'erifie les propri\'et\'es suivantes :
\begin{enumerate}
	\item Acyclicit\'e : Pour tout carr\'e acyclique 
	$$\begin{array}{ccc}
\ \widetilde{Y} & \rightarrow & \widetilde{X} \\
  \downarrow &      &  \downarrow  \\
  Y & \rightarrow & X 
  \end{array}$$	
dans $\mathbf{Sch}_c^G(\mathbb{R})$, le complexe filtr\'e simple du $\square_1^{+}$-diagramme dans $\mathcal{C}^G$
	$$\begin{array}{ccc}
\ {}^G\!\mathcal{W} C_{*}(\widetilde{Y}) & \rightarrow & {}^G\!\mathcal{W} C_{*}(\widetilde{X}) \\
  \downarrow &      &  \downarrow  \\
  {}^G\!\mathcal{W} C_{*}(Y) & \rightarrow & {}^G\!\mathcal{W} C_{*}(X) 
  \end{array}$$
  est acyclique (i.e. isomorphe au complexe nul dans $H o ~ \mathcal{C}^G$).
  	\item Additivit\'e : Pour une inclusion ferm\'ee \'equivariante $Y \subset X$, le complexe filtr\'e simple du $\square_0^{+}$-diagramme dans $\mathcal{C}^G$
	$${}^G\!\mathcal{W} C_{*}(Y) \rightarrow {}^G\!\mathcal{W} C_{*}(X)$$
	est isomorphe \`a ${}^G\!\mathcal{W} C_{*}(X \setminus Y)$.
\end{enumerate}
Un tel foncteur ${}^G\!\mathcal{W} C_{*}$ est unique \`a un isomorphisme de $H o ~ \mathcal{C}^G$ unique pr\`es.
\end{theo}

\begin{proof}

\textbf{\underline{Existence.}} Soit $X$ une $G$-vari\'et\'e alg\'ebrique r\'eelle. On a une action de $G$ sur le complexe filtr\'e de poids $\mathcal{W} C_*(X)$ de $X$ induite par fonctorialit\'e du complexe de poids.

De plus, tout morphisme propre r\'egulier \'equivariant induira, toujours par fonctorialit\'e de $\mathcal{W} C_*$, un morphisme de $H o ~ \mathcal{C}$ \'equivariant, i.e. un morphisme de $H o ~ \mathcal{C}^G$. 

On obtient donc ainsi un foncteur ${}^G\!\mathcal{W} C_* :  \mathbf{Sch}_c^G(\mathbb{R}) \rightarrow H o ~ \mathcal{C}^G$, celui qui \`a toute $G$-vari\'et\'e alg\'ebrique r\'eelle $X$ associe son complexe filtr\'e de poids $\mathcal{W} C_*(X)$ muni de l'action induite.
\\

Ce foncteur est bien une extension de $F^{can} C_{*} : \mathbf{V}^G(\mathbb{R}) \longrightarrow H o ~ \mathcal{C}^G$ car le complexe de poids $\mathcal{W} C_{*} : \mathbf{Sch}_c(\mathbb{R}) \longrightarrow H o ~ \mathcal{C}$ est une extension de $F^{can} C_{*} : \mathbf{V}(\mathbb{R}) \longrightarrow H o ~ \mathcal{C}$ (pour $X \in \mathbf{V}^G(\mathbb{R})$, l'action sur le complexe filtr\'e $F^{can} C_*(X)$ est \'egalement induite par fonctorialit\'e).
\\

Il v\'erifie enfin les conditions d'acyclicit\'e et d'additivit\'e (le morphisme nul est \'equivariant et, dans le th\'eor\`eme 2.2.2 de \cite{GNA}, la condition d'additivit\'e repose sur une \'egalit\'e permettant l'extension du foncteur aux vari\'et\'es non compactes, ce qui la rend a fortiori \'equivariante).
\\

\textbf{\underline{Unicit\'e.}} Pour l'unicit\'e, nous utiliserons le crit\`ere d'extension de Guill\'en-Navarro Aznar {\rm (\cite{GNA} Th\'eor\`eme 2.2.2)} adapt\'e au cadre avec action du groupe $G$ que l'on a \'enonc\'e pr\'ec\'edemment (th\'eor\`eme \ref{crit_action}). Appliquons-le (tout du moins sa version homologique et covariante) \`a notre contexte pour obtenir l'unicit\'e du complexe de poids avec action de $G$ :
\begin{itemize}
	\item $\mathcal{C}^G$ est une cat\'egorie de descente homologique, en tant que cat\'egorie des complexes de cha\^ines de $\mathbb{Z}_2[G]$-modules \`a gauche (i.e. les $\mathbb{Z}_2$-espaces vectoriels munis d'une action lin\'eaire du groupe $G$), born\'es et munis d'une filtration croissante born\'ee. En effet, les $\mathbb{Z}_2[G]$-modules \`a gauche forment une cat\'egorie ab\'elienne et on utilise alors la propri\'et\'e (1.7.5) de \cite{GNA}. 

	\item Le foncteur $F^{can} C_{*}$ est $\Phi$-rectifi\'e, car il est d\'efini sur la cat\'egorie $\mathcal{C}^G$.
	\item Le foncteur $F^{can} C_*$ v\'erifie les conditions (F1) et (F2) dans le cadre sans action. Munissant, dans la preuve de \cite{MCP} Theorem 1.1, les vari\'et\'es projectives lisses d'une action ag\'ebrique de $G$, tous les morphismes mentionn\'es, induits par l'application du foncteur $F^{can} C_*$, sont alors \'equivariants, et celui-ci v\'erifie donc bien les conditions (F1) et (F2) dans le cadre avec action de $G$. 

\end{itemize}

\end{proof}

\begin{de} Si $X$ est une $G$-vari\'et\'e alg\'ebrique r\'eelle, le $G$-complexe filtr\'e ${}^G\!\mathcal{W} C_*(X)$ (d\'efini \`a quasi-isomorphisme filtr\'e \'equivariant pr\`es) est appel\'e le complexe de poids avec action de $X$.
\end{de}

\begin{rem} \label{rem_fil_geom_act}
\begin{itemize} 
	\item L'isomorphisme $H_n({}^G\!\mathcal{W} C_*(X)) \cong H_n(X)$ pour toute $G$-vari\'et\'e r\'eelle $X$ (\cite{MCP} Proposition 1.5) est \'equivariant. En effet, dans la preuve de \cite{MCP} Proposition 1.5, les foncteurs $C_* : \mathbf{Sch}_c(\mathbb{R}) \rightarrow H o ~ \mathcal{D}$ et $\varphi \circ \mathcal{W} C_* : \mathbf{Sch}_c(\mathbb{R}) \rightarrow H o ~ \mathcal{D}$ sont isomorphes. Ainsi, si $X$ est une $G$-vari\'et\'e alg\'ebrique r\'eelle, le quasi-isomorphisme entre $\varphi \circ \mathcal{W} C_*(X)$ et $C_*(X)$ est \'equivariant par rapport aux actions induites par fonctorialit\'e.
	\item La filtration par le poids et la suite spectrale de poids se retrouvent munies de l'action induite du groupe $G$. 
	\item La filtration g\'eom\'etrique/Nash-constructible de McCrory et Parusi\'nski (\cite{MCP} sections 2 et 3), munie de l'action de $G$ induite par fonctorialit\'e, r\'ealise le complexe de poids avec action des $G$-vari\'et\'es alg\'ebriques r\'eelles. En effet, lorsque l'on munit les vari\'et\'es consid\'er\'ees d'une action alg\'ebrique de $G$, les morphismes des suites exactes courtes de \cite{MCP} Theorem 2.7 et Theorem 3.6 sont \'equivariants pour les actions induites, ainsi que les inclusions $\mathcal{G}_p C_k(X) \subset F_p^{can} C_k(X)$ pour toute $G$-vari\'et\'e r\'eelle $X$.
	 \item Comme dans le cadre sans action (\cite{MCP} Proposition 1.8), le complexe de poids ${}^G\!\mathcal{W} C_*(X)$ avec action de~$G$ d'une vari\'et\'e compacte non singuli\`ere est quasi-isomorphe dans $\mathcal{C}^G$ \`a $F^{can} C_*(X)$. Il en va de m\^eme pour toutes ses r\'ealisations.	
\end{itemize}

\end{rem}

Dans la suite, le complexe de poids avec action d'une $G$-vari\'et\'e r\'eelle $X$ sera simplement not\'e $\mathcal{W} C_*(X)$ lorsque le contexte sera explicite.

\section{Le d\'ecoupage d'une vari\'et\'e Nash affine compacte munie d'une action de $G = \mathbb{Z}/2\mathbb{Z}$} \label{section_decoup_var_nash}

On souhaite d\'ecouper toute cha\^ine invariante d'une vari\'et\'e alg\'ebrique r\'eelle avec action de $G = \mathbb{Z}/2\mathbb{Z} = \{1, \sigma\}$, apr\`es en avoir retir\'e la partie invariante point par point, en deux morceaux qui soient l'image l'un de l'autre par l'involution $\sigma$, mais des morceaux dont on puisse contr\^oler la r\'egularit\'e vis-\`a-vis de la filtration Nash-constructible.

On s'int\'eresse au cas particulier du groupe \`a deux \'el\'ements car il est le premier cas difficile, tant pour cette question particuli\`ere que pour l'\'etude des vari\'et\'es alg\'ebriques r\'eelles avec action de groupe en g\'en\'eral. 

L'ingr\'edient-cl\'e qui nous permettra d'effectuer ce d\'ecoupage sera le r\'esultat suivant. Ce th\'eor\`eme affirme que l'on peut d\'ecouper toute vari\'et\'e Nash affine compacte connexe munie d'une involution alg\'ebrique le long d'un sous-ensemble sym\'etrique par arcs (\cite{Kur}, \cite{KP}), globalement invariant sous l'action. Une sous-vari\'et\'e Nash d'un espace affine $\mathbb{R}^N$ est une sous-vari\'et\'e $\mathcal{C}^{\infty}$ de $\mathbb{R}^N$ qui en est \'egalement un sous-ensemble semi-alg\'ebrique. On renvoie \`a \cite{Shi} pour tout le mat\'eriel concernant les vari\'et\'es Nash et les fonctions Nash, i.e les fonctions $\mathcal{C}^{\infty}$ semi-alg\'ebriques.

\begin{theo} \label{decoup_Nash} Soit $M$ une sous-vari\'et\'e Nash compacte connexe d'un espace affine $\mathbb{R}^N$ munie d'une involution alg\'ebrique $\sigma$ (i.e. la restriction d'une involution alg\'ebrique sur l'adh\'erence de Zariski de M) non triviale. Alors il existe un sous-ensemble sym\'etrique par arcs $S$ de $M$ de codimension $1$ globalement stable sous l'action de $\sigma$, et un sous-ensemble semi-alg\'ebrique ferm\'e $A$ de $M$ tels que $M = A \cup \sigma(A)$, $S = A \cap \sigma(A)$ et $S = \partial A$.
\end{theo}

\begin{center}
\includegraphics{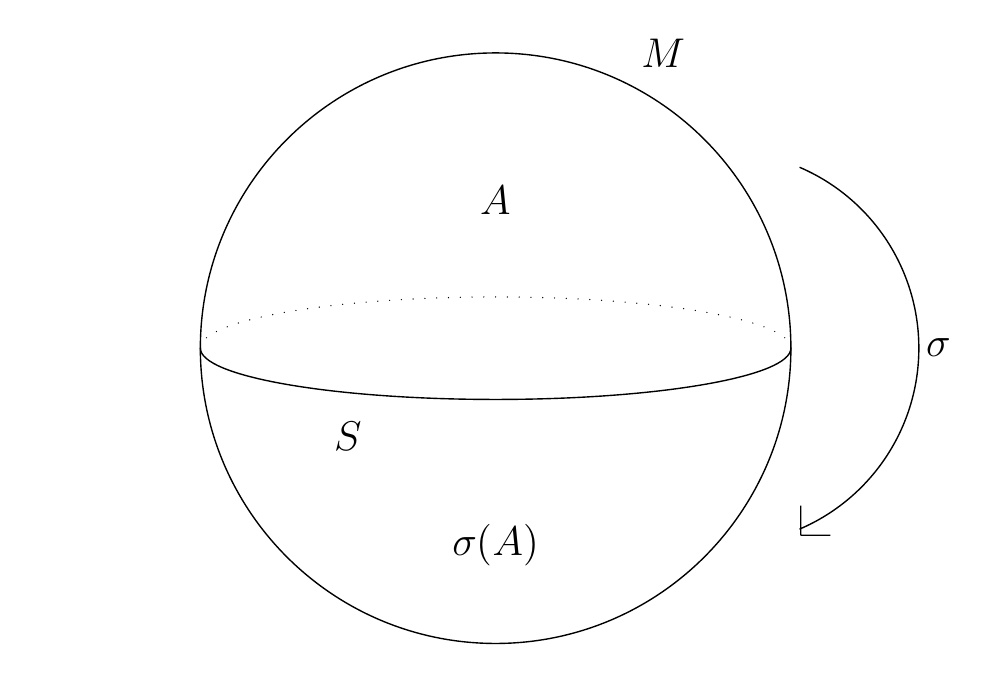}
\end{center} 

\vspace{0.5cm}
 
\begin{proof} La d\'emonstration de ce r\'esultat va se baser sur l'obtention d'une application classifiante entre le quotient de $M$ priv\'e de ses points fixes, que l'on notera $N$, par l'action de $G$ et un espace projectif r\'eel. On approchera cette fonction, tout du moins son extension \`a la compactificaction Nash de $N/G$, via le th\'eor\`eme de Stone-Weierstrass et l'utilisation d'un voisinage tubulaire Nash (\cite{BCR} Corollary 8.9.5), par une fonction Nash, transverse au sous-espace projectif de codimension $1$, que l'on rel\`evera ensuite en une fonction \'equivariante d\'efinie sur $M \setminus M^G$ et \`a valeurs dans une sph\`ere munie de l'action antipodale, Nash et transverse \`a la sph\`ere de codimension $1$. Consid\'erant alors l'image r\'eciproque de celle-ci, ainsi que celles des deux h\'emisph\`eres, auxquels on rajoute les points fixes, on parvient \`a d\'ecouper $M$ en deux sous-ensembles semi-alg\'ebriques, images l'une de l'autre par l'involution, le long d'un sous-ensemble sym\'etrique par arcs.
\\ 

\underline{(i) Quotient.} Consid\'erons donc $N$ la vari\'et\'e $M$ priv\'ee de ses points fixes. $N$ est une sous-vari\'et\'e Nash de $\mathbb{R}^N$, de m\^eme dimension que $M$ (car l'action de $G$ sur $M$ n'est pas triviale) et munie d'une action libre de $G$ (o\`u $G = \{id, \sigma \} = \mathbb{Z}/2\mathbb{Z}$). Le quotient $N/G$ est une vari\'et\'e Nash. En effet, le quotient d'un ensemble semi-alg\'ebrique par l'action (semi-alg\'ebrique) d'un groupe est un ensemble semi-alg\'ebrique (\cite{PS}), et le quotient d'une vari\'et\'e (diff\'erentiable) lisse (i.e. $\mathcal{C}^{\infty}$) par l'action libre (et lisse) d'un groupe fini est une vari\'et\'e lisse (\cite{Lee} Theorem 7.10). La vari\'et\'e Nash $N/G$ est de plus affine, en tant que sous-ensemble (semi-alg\'ebrique) de l'ensemble des points r\'eels du quotient de la complexification de l'adh\'erence de Zariski de $M$, qui est une vari\'et\'e complexe projective, par la complexification de l'action de $G$, quotient qui est lui aussi une vari\'et\'e projective complexe. 
\\

\underline{(ii) Compactification.} La vari\'et\'e $N/G$ \'etant une vari\'et\'e Nash affine, on peut la compactifier en une vari\'et\'e Nash affine avec bord d'apr\`es le th\'eor\`eme VI.2.1. de \cite{Shi}. Plus pr\'ecis\'ement, il existe une vari\'et\'e alg\'ebrique affine compacte non-singuli\`ere $X'$, une sous-vari\'et\'e alg\'ebrique non-singuli\`ere $Y'$ de codimension $1$ (vide si $N/G$ est compacte i.e. si $M^G = \emptyset$), une composante connexe $N'$ de $X \setminus Y$ telles qu'il existe un diff\'eomorphisme Nash $\psi : N/G \rightarrow N'$ et $M' := cl(N')$ est une vari\'et\'e Nash \`a bord compacte (connexe), de bord $Y$. 
\\

\underline{(iii) Application classifiante.} Le groupe $G = \mathbb{Z}/2\mathbb{Z}$ agissant librement sur le vari\'et\'e lisse~$N$, on peut consid\`erer le $\mathbb{Z}/2 \mathbb{Z}$-fibr\'e principal lisse $\Pi : N \rightarrow N/G$ ainsi qu'une application classifiante lisse $f : N/G \rightarrow \mathbb{P}^M(\mathbb{R})$ (\cite{Hus}). Remarquons que ``\,l'\,'' application classifiante est d\'efinie \`a homotopie pr\`es. Aussi, notre but va \^etre d\`es \`a pr\'esent de construire une application classifiante $h : N/G \rightarrow \mathbb{P}^{M}(\mathbb{R})$ qui sera Nash et transverse \`a $\mathbb{P}^{M-1}(\mathbb{R})$.
\\

\underline{(iv) Extension.} On commence pour cela par composer $f$ par $\psi^{-1}$ : on obtient une application lisse $f' := f \circ \psi^{-1} : N' \rightarrow \mathbb{P}^M(\mathbb{R})$. Puis on \'etend \`a homotopie pr\`es l'application $f'$ \`a la compactification Nash $M'$ de $N'$, ceci gr\^ace \`a une propri\'et\'e \'enonc\'ee par M. Shiota dans \cite{Shi}. L'application $f'$ \'etant en effet une application continue entre $N' = M' \setminus \partial M'$ et $\mathbb{P}^M(\mathbb{R})$, o\`u $M'$ et $\mathbb{P}^M(\mathbb{R})$ sont en particulier des vari\'et\'es PL (piecewise-linear) compactes, le lemme V.1.4 de \cite{Shi} nous dit que $f'$ est homotope \`a la restriction \`a $N'$ d'une application PL, et donc en particulier continue semi-alg\'ebrique, $\overline{f'} : M' \rightarrow \mathbb{P}^M(\mathbb{R})$. On va par la suite chercher \`a approcher la fonction $\overline{f'}$ gr\^ace \`a un voisinage tubulaire Nash.
\\

\underline{(v) Approximations.} Plongeons donc $M'$ dans un espace affine $\mathbb{R}^{N_0}$ et $\mathbb{P}^M(\mathbb{R})$ dans un espace affine $\mathbb{R}^{M_0}$. On consid\`ere un voisinage tubulaire Nash $(U, \rho : U \rightarrow \mathbb{P}^M(\mathbb{R}))$ de $\mathbb{P}^M(\mathbb{R})$ dans $\mathbb{R}^{M_0}$ (\cite{BCR} Corollary 8.9.5). 

$M'$ \'etant compact, on approche alors $h_0 := \overline{f'}$ par une application polynomiale $h_1$ en utilisant le th\'eor\`eme de Stone-Weirstrass (\cite{BCR} Theorem 8.8.5). Puis on approche l'application $h_1$ par une application lisse $h_2$ transverse \`a  $\mathbb{P}^{M-1}(\mathbb{R})$. Enfin, cette derni\`ere condition \'etant ouverte, on approche $h_2$ par une application polynomiale (en utilisant encore une fois Stone-Weierstrass) de fa\c{c}on suffisamment proche pour obtenir une application $h_3$ qui soit polynomiale et transverse \`a $\mathbb{P}^{M-1}(\mathbb{R})$. Pr\'ecisons que l'on effectue ces approximations de mani\`ere \`a ce qu'\`a chaque fois, pour tout $t \in [0,1]$ et pour tout $x \in M/G$, $(1-t)h_i(x) + t h_{i+1}(x) \in U$.

Consid\'erant alors les homotopies successives
$$(t,x) \mapsto \rho((1-t) h_i(x) + t h_{i+1}(x)),$$
on obtient ainsi une homotopie entre $\overline{f'} : M/G \rightarrow \mathbb{P}^M(\mathbb{R})$ et $\overline{h'} := \rho \circ h_3 : M/G \rightarrow \mathbb{P}^M(\mathbb{R})$. L'application $\overline{h'}$ est Nash et transverse \`a $\mathbb{P}^{M-1}(\mathbb{R})$, car $\rho$ est une submersion Nash (en tant que composition d'un diff\'eomorphisme Nash et de la projection d'un fibr\'e vectoriel Nash).
\\

\underline{(vi) Restriction.} Enfin, en restreignant l'homotopie \`a $N'$ et en notant $h'$ la restriction de $\overline{h'}$ \`a $N'$, on obtient une application 
$$h' : N' \rightarrow \mathbb{P}^M(\mathbb{R})$$
Nash, transverse \`a $\mathbb{P}^{M-1}(\mathbb{R})$, homotope \`a $f'$. En la composant enfin avec $\psi$ pour revenir au quotient $N/G$, on peut supposer que l'application
$$h := h' \circ \psi : N/G \rightarrow \mathbb{P}^M(\mathbb{R})$$
Nash, transverse \`a $\mathbb{P}^{M-1}(\mathbb{R})$ et homotope \`a $f = f' \circ \psi$, est l'application classifiante associ\'ee au $\mathbb{Z}/2 \mathbb{Z}$-fibr\'e principal lisse $\Pi : N \rightarrow N/G$.

\underline{(vii) Rel\`evement.} On rel\`eve alors $h$ en une application
$$\widetilde{h} : N \rightarrow \mathbb{S}^M(\mathbb{R})$$
\'equivariante (par rapport \`a l'action de $G$ sur $N$ et l'action antipodale sur $\mathbb{S}^M(\mathbb{R})$) transverse \`a $\mathbb{S}^{M-1}(\mathbb{R})$, analytique et localement semi-alg\'ebrique, donc Nash, car une application analytique localement semi-alg\'ebrique entre deux vari\'et\'es Nash, dont celle d'arriv\'ee est affine, est Nash (remarque (xv) \`a la page 16 de \cite{Shi}). 
\\

\underline{(viii) D\'ecoupage de $N$.} On note 
\begin{itemize}
	\item $N^+ := \widetilde{h}^{-1}(S^+)$, 
	\item $N^- := \widetilde{h}^{-1}(S^-) = \sigma(N^+)$, 
	\item $W := \widetilde{h}^{-1}\left(\mathbb{S}^{M-1}(\mathbb{R})\right) = N^+ \cap N^-$.
\end{itemize}
($S^+$ et $S^-$ sont respectivement les h\'emisph\`eres nord et sud de la sph\`ere $\mathbb{S}^{M}(\mathbb{R})$, \'echang\'es sous l'action antipodale).

$N^+$ et $N^-$ sont des sous-ensembles semi-alg\'ebriques ferm\'es de $N$ de m\^eme dimension que $N$, et, par transversalit\'e de l'application Nash $\widetilde{h}$ (et parce que $\mathbb{S}^{M-1}(\mathbb{R})$ est une sous-vari\'et\'e Nash de $\mathbb{S}^{M}(\mathbb{R})$), $W$ est une sous-vari\'et\'e Nash de $N$ de codimension $1$. C'est de plus un sous-ensemble $\mathcal{AS}$ de $N$, en tant qu'image r\'eciproque du sym\'etrique par arcs $\mathbb{S}^{M-1}(\mathbb{R})$ par l'application analytique semi-alg\'ebrique $\widetilde{h}$.
\\
	
On a bien \'evidemment $N^+ \cup N^- = N$ mais on a \'egalement $\partial N^+ = \partial N^- = W$. 

En effet, le bord semi-alg\'ebrique d'un ensemble semi-alg\'ebrique $A$ est tout d'abord d\'efini par 
$$\partial A = \{ x \in A ~|~\chi(A \cap S(x,\epsilon)) \equiv 1 \mbox{ mod $2$ } \}$$
pour $\epsilon$ suffisamment petit, or, la caract\'eristique d'Euler \`a supports compacts \'etant additive et $N$ et $W$ ne poss\'edant pas de bord, on a modulo $2$, pour tout $x$ dans $N$,
\begin{eqnarray*}
\	0 & \equiv & \chi\left(N \cap S \left(x,\epsilon\right)\right) \\
& \equiv &  \chi\left( N^+ \cap S\left(x,\epsilon\right)\right) + \chi\left(N^- \cap S\left(x,\epsilon\right)\right) - \chi\left( W \cap S\left(x,\epsilon\right)\right) \\
& \equiv & \chi\left( N^+ \cap S\left(x,\epsilon\right)\right) + \chi\left(N^- \cap S\left(x,\epsilon\right)\right)
\end{eqnarray*}
Un \'el\'ement $x$ de $N$ est ainsi dans le bord de $N^+$ si et seulement s'il est dans le bord de $N^-$. En particulier, $\partial N^+ = \partial N^- \subset N^+ \cap N^- = W$.

R\'eciproquement, soit $x \in W$ et supposons par l'absurde que $x \notin \partial N^+ = \partial N^-$. Le point $x$ appartient alors \`a $N^+ \setminus \partial N^+$ et $N^- \setminus \partial N^-$, qui sont des ouverts semi-alg\'ebriques de $N$ : il existe donc des ouverts semi-alg\'ebriques $U$ et $V$ de $N$ (de m\^eme dimension $n$ que $N$) tels que $x \in U \subset N^+ \setminus \partial N^+$, et  $x \in V \subset N^- \setminus \partial N^-$. L'ouvert $U \cap V$ de $N$, de dimension $n$ (non vide car contenant $x$), est alors inclus dans $(N^+ \setminus \partial N^+) \cap (N^- \setminus \partial N^-)$ qui est de dimension au plus $n-1$, en tant que sous-ensemble de $W$.
\\

\underline{(ix) D\'ecoupage de $M$.} Enfin, on rajoute les points fixes en consid\'erant l'adh\'erence de $N$ dans $M$, qui est $M$ tout entier car $M$ est une vari\'et\'e Nash connexe dont la sous-vari\'et\'e Nash $M^G$ est de codimension au moins $1$. On note alors 
\begin{itemize}
	\item $M^+ := cl(N^+)$,
	\item $M^- := cl(N^-) = \sigma(M^+)$.
\end{itemize}

On voit imm\'ediatement que $M^+ \cup M^- = M$. Montrons que l'on a $M^+ \cap M^- = W \cup M^G$. On prouve pour cela que $cl(N^+) \setminus N^+ = cl(N^-) \setminus N^- = W \cup M^G$.
\\

En effet, comme $cl(N^+) \setminus N^+ \subset M^G$, 
$$cl(N^+) \setminus N^+ = \sigma \left(cl(N^+) \setminus N^+ \right) = cl(N^-) \setminus N^-.$$   
De plus,
$$M = M^+ \cup M^- = (N^+ \cup N^-) \sqcup \left(\left(cl(N^+) \setminus N^+\right) \cup \left(cl(N^-) \setminus N^- \right)\right) = (M \setminus M^G) \sqcup \left(cl(N^+) \setminus N^+\right)$$
donc $cl(N^+) \setminus N^+ = M^G$. En particulier, $M^+ \cap M^- = W \cup M^G$. Remarquons que, comme $W$ est un sous-ensemble $\mathcal{AS}$ de $M \setminus M^G$, $W \cup M^G$ est un sous-ensemble $\mathcal{AS}$ de $M$, en tant qu'union de deux tels ensembles (la cat\'egorie $\mathcal{AS}$ est l'alg\`ebre bool\'eenne engendr\'e par les sous-ensembles sym\'etriques par arcs de $\mathbb{P}^N(\mathbb{R})$). Il est de plus ferm\'e dans le compact $M$ et est donc sym\'etrique par arcs.
\\

Enfin, on montre que $\partial M^+ = \partial M^- = W \cup M^G$, d'une mani\`ere identique \`a ce que l'on avait fait dans le cas de $N$.

\end{proof}

Dans la section suivante \ref{suite_exacte_smith_nash_cons}, on va, comme annonc\'e, utiliser ce r\'esultat pour pouvoir d\'ecouper toute cha\^ine d'une vari\'et\'e alg\'ebrique r\'eelle en deux morceaux suffisamment r\'eguliers. 

\section{La suite exacte courte de Smith Nash-constructible dans le cas $G = \mathbb{Z}/2 \mathbb{Z}$} \label{suite_exacte_smith_nash_cons}

Soit $X$ une vari\'et\'e alg\'ebrique r\'eelle munie d'une involution alg\'ebrique $\sigma$. Dans cette partie, on montre que l'on peut \'ecrire toute cha\^ine $c$ globalement invariante de $X$, de dimension $k$ et de degr\'e $\alpha$ dans la filtration Nash-constructible (avec action), comme la somme
$$c = c_{|X^G} + (1+ \sigma) \gamma$$
o\`u $c_{|X^G} \in \mathcal{N}_{\alpha}ÊC_k(X^G)$ est la restriction de $c$ \`a l'ensemble des points invariants de $X$ et o\`u $\gamma \in \mathcal{N}_{\alpha+1} C_k(X)$.
\\

Consid\'erant une cha\^ine $c$, notre d\'emarche va consister \`a d\'ecouper en de tels morceaux l'adh\'erence de Zariski du support de $c$, puis \`a consid\'erer les intersections du support de $c$ avec ceux-ci, qui r\'epondront alors aux conditions recherch\'ees.   

On proc\`ede ainsi en deux \'etapes. On traite tout d'abord le cas d'une cha\^ine pure, et ensuite celui d'une cha\^ine quelconque.
\\

Pour la premi\`ere \'etape, on se ram\`enera au cas o\`u le support de la cha\^ine est une vari\'et\'e Nash compacte et on utilisera le th\'eor\`eme \ref{decoup_Nash} de la section pr\'ec\'edente, afin de d\'ecouper une cha\^ine pure de $X$, invariante sous l'action de $G$, en deux cha\^ines dont le bord est pur :

\begin{prop} Soit $c \in \left(\mathcal{N}_{-k} C_k(X)\right)^G$ une cha\^ine pure de dimension $k$, invariante sous l'action de $G$. Alors il existe une cha\^ine $\gamma \in \mathcal{N}_{-k + 1} C_k(X)$ telle que
$$c = c_{|X^G} + (1+\sigma) \gamma$$
(la restriction $c_{|X^G}$ \`a $X^G$  appartient \`a $\mathcal{N}_{-k} C_k(X^G)$).
\end{prop}

\begin{proof} Notons tout d'abord que, $c$ \'etant une cha\^ine invariante sous l'action de $G$, le support de $c$ est \'egalement globalement invariant. En effet, si $A$ est un sous-ensemble semi-alg\'ebrique de $X$ repr\'esentant la cha\^ine $c$, c'est \'egalement le cas pour $\sigma(A)$ et, comme l'involution $\sigma$ est un isomorphisme alg\'ebrique, on a
$$\Supp~c = \{y \in \sigma(A)~|~\dim_y~ \sigma(A) = k \} = \sigma \left(\{ x \in A~|~\dim_x A = k\} \right) = \sigma( \Supp~c).$$

On va alors se ramener au cas o\`u le support de $c$, que l'on note $M$, est une sous-vari\'et\'e Nash compacte connexe d'un espace affine.
\\

(1) On peut supposer que la $G$-vari\'et\'e alg\'ebrique r\'eelle $X$ est compacte~: si ce n'est pas le cas, on consid\`ere une compactification \'equivariante $X_0$ de $X$ (\cite{DL} Appendix), et l'adh\'erence $\overline{c} \in \left(\mathcal{N}_{-k} C_k(X_0)\right)^G$ de $c$ dans $X_0$. Si l'on montre alors que, dans ces conditions, il existe $\gamma_0 \in \mathcal{N}_{-k+1} C_k(X_0)$ tel que $\overline{c} = \overline{c}_{|X_0^G} + (1+\sigma) \gamma_0$, on obtient par restriction
$$c = \overline{c}_{|X} = c_{|X^G} + (1+\sigma) {\gamma_0}_{|X}$$
avec ${\gamma_0}_{|X} \in \mathcal{N}_{-k+1} C_k(X).$
\\

(2) On suppose ensuite que $X$ est l'adh\'erence de Zariski $\overline{M}^{\mathcal{Z}}$ de $M$ et qu'en particulier $\dim X = k$. En effet, si l'on note $Z := \overline{M}^{\mathcal{Z}}$, alors $Z$ est un ferm\'e de $X$ globalement invariant sous l'action de $G$ (car $M$ l'est) et, comme par d\'efinition $\Supp~c \subset Z$, on a $c \in \mathcal{N}_{-k} C_k(Z)$.
\\

(3) On peut aussi supposer que la vari\'et\'e $X$ est non-singuli\`ere, en consid\'erant une r\'esolution \'equivariante $\pi : \widetilde{X} \rightarrow X$ des singularit\'es de $X$ (\cite{Hiro}, \cite{DL} Appendix) et le tir\'e en arri\`ere $\widetilde{c}:= \pi^{-1} c \in \left(\mathcal{N}_{-k} C_k(\widetilde{X})\right)^G$ de $c$. Alors, si l'on montre que dans ces conditions, $\widetilde{c} = \widetilde{c}_{|\widetilde{X}^G} + (1+ \sigma) \widetilde{\gamma}$ avec $\widetilde{\gamma} \in \mathcal{N}_{-k+1} C_k(\widetilde{X})$, on obtient, en poussant en avant,
$$c = \pi_*(c) = c_{|X^G} + (1+ \sigma) \pi_*(\widetilde{\gamma})$$ 
(avec $\pi_*(\widetilde{\gamma}) \in  \mathcal{N}_{-k+1} C_k(X)$), car $\pi$ est un isomorphisme en dehors d'une sous-vari\'et\'e de codimension au moins $1$ et $c$ est de dimension $k = \dim X$).
\\

Au total, la cha\^ine $c$ \'etant pure, on peut supposer que le support $M$ de $c$ est un sous-ensemble sym\'etrique par arcs (dont toutes les composantes connexes sont) de dimension maximale d'une vari\'et\'e compacte non-singuli\`ere $X$, i.e. une union de composantes connexes de dimension maximale de $X$ (\cite{KP}).
\\

(4) Enfin, l'involution $\sigma$ \'echange ou fixe globalement les cartes affines recouvrant $X$, ainsi que les composantes connexes de $X$, et donc les restrictions de $M$ \`a celles-ci, qui constituent \'egalement des cha\^ines pures de $X$. On peut donc supposer sans perdre de g\'en\'eralit\'e que $M$ est une composante connexe d'un sous-ensemble alg\'ebrique compact non-singulier d'un espace affine, globalement stable sous l'action (non triviale) de $\sigma$, qui peut \^etre consid\'er\'ee comme une sous-vari\'et\'e Nash compacte connexe d'un espace affine.
\\

On utilise alors le th\'eor\`eme \ref{decoup_Nash} pour \'ecrire
$$[M] = (1 + \sigma) [A]$$
dans $C_k(X)$ (l'ensemble des points invariants $M^G$ est de codimension au moins $1$ comme l'action de $G$ n'est pas triviale, donc $[M^G] = 0$ dans $C_k(X)$), avec $[A] \in C_k(X)$ v\'erifiant $\partial [A] = [\partial(A)] \in \mathcal{N}_{-k+1} C_{k-1}(X)$, i.e. $[A] \in \mathcal{N}_{-k+1} C_{k}(X)$ (on utilise l'une des \'equivalences de la preuve de \cite{MCP} Corollary 3.12 : $X$ est compacte non-singuli\`ere de dimension $k$).

\end{proof}

On va alors utiliser le d\'ecoupage des cha\^ines pures pour d\'ecouper les cha\^ines de degr\'e plus \'elev\'e dans la filtration. Pr\'ecis\'ement, pour couper en deux une cha\^ine invariante sous l'action de $G$, de fa\c{c}on \`a contr\^oler la ``r\'egularit\'e'' des parties que l'on obtient, il suffit de d\'ecouper l'adh\'erence de Zariski de son support en utilisant la propri\'et\'e pr\'ec\'edente : 

\begin{prop} \label{decoup_fin}
Soit $c \in (\mathcal{N}_{\alpha} C_k(X))^G$ une $k$-cha\^ine quelconque de $X$, invariante sous l'action de l'involution $\sigma$. Alors on peut \'ecrire
$$c = c_{|X^G} + (1 + \sigma) c'$$
avec $c_{|X^G} \in \mathcal{N}_{\alpha} C_k(X^G)$ et $c' \in \mathcal{N}_{\alpha + 1} C_k(X)$
\end{prop}

\begin{proof}  On note $c^{\mathcal{Z}}$ la cha\^ine de $C_k(X)$ repr\'esent\'ee par l'adh\'erence de Zariski du support not\'e $A$ de la cha\^ine $c$. On a alors $c^{\mathcal{Z}} \in \left(\mathcal{N}_{-k} C_k(X)\right)^G$ et, d'apr\`es la propri\'et\'e pr\'ec\'edente, il existe $\gamma \in \mathcal{N}_{-k+1} C_k(X)$ tel que
$$c^{\mathcal{Z}} = \left(c^{\mathcal{Z}}\right)_{|X^G} + (1 + \sigma) \gamma.$$

Revenons \`a la d\'efinition de la filtration Nash-constructible (\cite{MCP} section 3) : soit donc une fonction $\psi : X \rightarrow~2\mathbb{Z}$ g\'en\'eriquement Nash-constructible en dimension $k$ telle que $\gamma = [S]$ avec
$$S = \{x \in X~|~\psi (x) \notin 2^{2} \mathbb{Z} \},$$
et soit une fonction $\varphi : X \rightarrow 2^{k+\alpha} \mathbb{Z}$ g\'en\'eriquement Nash-constructible en dimension $k$ telle que $c = [B]$ avec
$$B = \{x \in X~|~\varphi (x) \notin 2^{k + \alpha +1} \mathbb{Z} \}.$$

On consid\`ere alors le produit $\varphi \times \psi : X \rightarrow 2^{k + \alpha + 1} \mathbb{Z}$ qui est une fonction g\'en\'eriquement Nash constructible en dimension $k$ (comme produit de telles fonctions). On a
\begin{eqnarray*}
\ B \cap S & = & \{x \in X~|~\varphi(x) \notin 2^{k + \alpha +1} \mathbb{Z} \mbox{ et }   \psi (x) \notin 2^{2} \mathbb{Z} \} \\
& = & \{x \in X~|~\varphi \times \psi(x) \notin 2^{k + \alpha + 2} \mathbb{Z} \}.
\end{eqnarray*}

On note $c'$ la cha\^ine repr\'esent\'ee par $B \cap S$, qui appartient donc \`a $\mathcal{N}_{\alpha + 1} C_k(X)$, et on a
\begin{eqnarray*}
\ (1+ \sigma) c' & = & [B \cap S] + [\sigma(B) \cap \sigma(S)] \\
			& = & [B \cap S] + [B \cap \sigma(S)] \text{ (lemme \ref{lem_1_2_3})} \\
			& = & [cl\left( B \cap \left(S \div \sigma(S)\right)\right)]\\
			& = &  [B \cap cl\left(S \div \sigma(S)\right)] \text{ (lemme \ref{lem_1_2_cl})},
\end{eqnarray*}
or $[B] = [A]$ et
$$[cl\left(S \div \sigma(S)\right)] = (1 + \sigma) [S] = c^{\mathcal{Z}} + \left(c^{\mathcal{Z}}\right)_{|X^G} = \left[cl\left(\overline{A}^{\mathcal{Z}} \setminus \overline{A}^{\mathcal{Z}} \cap X^G\right)\right],$$
donc
\begin{eqnarray*}
\ (1 + \sigma) c' & = & \left[A \cap cl \left(\overline{A}^{\mathcal{Z}} \setminus \overline{A}^{\mathcal{Z}} \cap X^G\right) \right] \text{ (lemme \ref{lem_1_2_3})}\\
& = & \left[cl\left(A \cap \left(\overline{A}^{\mathcal{Z}} \setminus \overline{A}^{\mathcal{Z}} \cap X^G\right) \right)\right] \text{ (lemme \ref{lem_1_2_cl})}\\
& = & \left[cl\left(A \setminus A \cap X^G\right)\right] \\
& = & c + c_{|X^G}.
\end{eqnarray*}

\end{proof}

On r\'esume cette propri\'et\'e de d\'ecoupage des cha\^ines invariantes par une suite exacte courte qui rappelle la suite exacte courte de Smith, que l'on adapte ici aux contraintes de la filtration Nash-constructible. On d\'efinit pour cela une nouvelle filtration avant de donner la suite exacte courte que l'on nommera Smith Nash-constructible :

\begin{de} Pour tout $k$ et tout $\alpha$, on note
$$T^{\alpha + 1}_k(X) := \{c \in \mathcal{N}_{\alpha+1} C_k(X)~|~(1 + \sigma) c \in \mathcal{N}_{\alpha} C_k(X) \}.$$
\end{de}

\begin{rem} Pour tout $\alpha$, $T^{\alpha+1}_*(X)$, muni de la diff\'erentielle induite par celle de $C_*(X)$, est un complexe de cha\^ines. On obtient ainsi une nouvelle filtration
$$0 = T^{-k-1}_k(X) \subset T^{-k}_k(X) \subset \cdots \subset T^{1}_k(X) = C_k(X)$$
de $C_*(X)$.
\end{rem}

\begin{theo} \label{suite_ex_smith_nash} Pour tout $\alpha$, la suite de complexes
$$0 \rightarrow \mathcal{N}_{\alpha} C_*(X^G) \oplus (1+ \sigma) T^{\alpha + 1}_*(X) \rightarrow \mathcal{N}_{\alpha} C_*(X) \rightarrow (1 + \sigma) \mathcal{N}_{\alpha} C_*(X) \rightarrow 0,$$
est exacte. 

On l'appelle suite exacte courte de Smith Nash-constructible de degr\'e $\alpha$.
\end{theo}

On peut interpr\'eter cette suite exacte dans le cas o\`u $X$ est une vari\'et\'e compacte munie d'une action libre de $G$. En effet, dans ces conditions, on a un isomorphisme filtr\'e entre les cha\^ines invariantes et les cha\^ines du quotient~:

\begin{prop} \label{fil_Nash_quotient} Supposons que la $G$-vari\'et\'e alg\'ebrique r\'eelle $X$ soit compacte, et que l'action de $G$ sur $X$ soit libre. Alors le quotient $X/G$ (qui d\'esigne par abus de notation le quotient de l'ensemble des points r\'eels de $X$ par l'action de $G$ restreinte) est un ensemble sym\'etrique par arcs (\cite{GF} Proposition 3.15) et, pour tous $k, \alpha \in \mathbb{Z}$, on a un isomorphisme 
$$\left(\mathcal{N}_{\alpha} C_k(X)\right)^G \cong (1+ \sigma) T^{\alpha + 1}_k(X) \cong \mathcal{N}_{\alpha} C_k\left(X/G\right).$$
\end{prop}

\begin{proof} Fixons $k$ et $\alpha$. Notons tout d'abord que la suite exacte de Smith Nash-construc-tible nous donne le premier isomorphisme, \'etant donn\'e que l'action de $G$ sur $X$ est libre. 

Pour montrer le second isomorphisme, on consid\`ere d'une part le morphisme
$$1 + \sigma : T^{\alpha + 1}_k(X) \rightarrow T^{\alpha + 1}_k(X)$$
dont le noyau est $\left(T^{\alpha + 1}_k(X)\right)^G$ et l'image est $(1 + \sigma) T^{\alpha + 1}_k(X)$, et d'autre part le morphisme
$$T^{\alpha+1}_k(X) \xrightarrow{\pi_*} \mathcal{N}_{\alpha + 1} C_k\left(X/G\right),$$
obtenu par restriction \`a $T^{\alpha+1}_k(X)$ du morphisme $\mathcal{N}_{\alpha+1} C_k(X) \rightarrow \mathcal{N}_{\alpha + 1} C_k(X/G)$, induit par l'application quotient $X \rightarrow X/G$ (qui est une application propre continue de graphe $\mathcal{AS}$).

On montre par la suite que le noyau de $\pi_*$ est $\left(T^{\alpha +1}_k(X)\right)^G$ et que son image est $\mathcal{N}_{\alpha} C_k\left(X/G\right)$. Ainsi, les morphismes $1+ \sigma$ et $\pi_*$ poss\'edant les m\^emes noyaux (et les m\^emes espaces de d\'epart), on obtient un isomorphisme entre leurs images, soit
$$(1+ \sigma) T^{\alpha + 1}_k(X) \cong \mathcal{N}_{\alpha} C_k\left(X/G\right).$$
\\

Montrons donc dans un premier temps que le noyau de $\pi_*$ consiste en les cha\^ines invariantes de $T^{\alpha +1}_k(X)$ (qui sont les exactement les cha\^ines invariantes de $\mathcal{N}_{\alpha+1} C_k(X)$).
\\

Soit $c \in T_{\alpha +1}$ tel que $\pi_* c = 0$. Si $c = \Supp~\varphi_{c,\alpha}~\modulo~2^{k+\alpha+2}$ avec $\varphi_{c,\alpha} : X \rightarrow 2^{k+ \alpha +1} \mathbb{Z}$ g\'en\'eriquement Nash-constructible en dimension $k$, alors
$$\pi_* c = \Supp~\pi_*(\varphi_{c,\alpha})~\modulo~2^{k+\alpha+2}$$
o\`u $\pi_*(\varphi_{c,\alpha}) : X/G \rightarrow 2^{k+ \alpha +1} \mathbb{Z}$ est une fonction g\'en\'eriquement Nash-constructible en dimension $k$ sur $X/G$ (\cite{MCP} Corollary 3.5).

Or, pour toute fonction Nash-constructible $f$ sur $X$, comme l'action de $\sigma$ sur $X$ est libre, $\pi_*(f)(\overline{x}) = f(x) + f(\sigma(x))$ pour tout point $\overline{x} \in X/G$. Ainsi,
$$ \pi_*c = \{ \overline{x} \in X/G~|~\varphi(x) + \varphi(\sigma(x)) \notin 2^{k+ \alpha +2} \mathbb{Z} \},$$
avec $\varphi := \varphi_{c,\alpha}$.

Par hypoth\`ese cette cha\^ine est nulle, i.e. pour tout $\overline{x}$ en dehors d'un sous-ensemble semi-alg\'ebrique de $X/G$ de dimension $<k$ (i.e. pour tout $x$ en dehors d'un sous-ensemble semi-alg\'ebrique de $X$ de dimension $<k$), $\varphi(x) + \varphi(\sigma(x)) \in 2^{k+\alpha+2} \mathbb{Z}$. En particulier, en dehors d'un sous-ensemble semi-alg\'ebrique de codimension au moins $1$, tout point du repr\'esentant $\{x \in X~|~\varphi(x) \notin 2^{k+\alpha +2} \mathbb{Z} \}$ de la cha\^ine $c$ v\'erifie \`a la fois $\varphi(x) \notin 2^{k+ \alpha+ 2} \mathbb{Z}$ (par d\'efinition) et $\varphi(x) + \varphi(\sigma(x)) \in 2^{k+\alpha+2} \mathbb{Z}$, et donc \'egalement $\varphi(\sigma(x)) \notin 2^{k+\alpha+2} \mathbb{Z}$. Or $\varphi \circ \sigma = \sigma^*\left( \varphi \right)$ repr\'esente la cha\^ine $\sigma c$, et la cha\^ine $c$ est donc invariante sous l'action de $\sigma$. 
\\

On a donc $ker~\pi_* \subset (T^{\alpha+1}_k(X))^G$. R\'eciproquement, si $c \in (T^{\alpha+1}_k(X))^G$, alors, en reprenant les notations ci-dessus, $\varphi(x) \notin 2^{k+\alpha+2} \mathbb{Z}$ et $\varphi(\sigma(x)) \notin 2^{k+\alpha+2} \mathbb{Z}$ pour tout $x$ dans le support de~$c$, g\'en\'eriquement en dimension $k$, et donc $\varphi(x) + \varphi(\sigma(x)) \in 2^{k+\alpha+2} \mathbb{Z}$. Ainsi, $\pi_* c = 0$
\\

Montrons maintenant que l'image de $\pi_*$ est constitu\'ee des cha\^ines de degr\'e $\alpha$ du quotient~$X/G$.
\\

Soit $c \in \mathcal{N}_{\alpha} C_k(X/G)$ et soit $\varphi : X/G \rightarrow 2^{k+\alpha} \mathbb{Z}$ g\'en\'eriquement Nash-constructible en dimension $k$ telle que
$$c = \{\overline{x} \in X/G~|~\varphi(\overline{x}) \notin 2^{k+\alpha+1} \mathbb{Z} \}.$$
On consid\`ere la cha\^ine, que l'on note $\pi^* c$, repr\'esent\'ee par l'ensemble semi-alg\'ebrique de dimension $k$
$$\{x \in X~|~\pi^*(\varphi) (x) \notin 2^{k + \alpha +1} \mathbb{Z} \} = \pi^{-1}\left(\{\overline{x} \in X/G~|~\varphi(\overline{x}) \notin 2^{k+\alpha+1} \mathbb{Z} \} \right)$$
o\`u le tir\'e en arri\`ere $\pi^*(\varphi) : X \rightarrow 2^{k + \alpha}Ê\mathbb{Z}$ (\cite{MCP} Corollary 3.5) est une fonction g\'en\'eriquement Nash-constructible en dimension~$k$ telle que, pour $x \in X$ en dehors d'un sous-ensemble semi-alg\'ebrique de dimension $<k$, $\pi^*(\varphi)(x) = \varphi(\pi(x)) = \varphi(\overline{x})$.

La cha\^ine $\pi^*c$ appartient donc \`a $\mathcal{N}_{\alpha} C_k(X)$ et est de plus invariante sous l'action de $G$. En suivant le raisonnement qui nous a men\'es \`a l'exactitude de la suite de Smith Nash-constructible, on sait que l'on peut \'ecrire $\pi^* c = (1 + \sigma) \gamma$ avec 
$$\gamma = \{x \in X~|~\psi(x) \notin 2^2 \mathbb{Z} \mbox{ et } \pi^*(\varphi) (x) \notin 2^{k + \alpha +1} \mathbb{Z} \}$$ 
o\`u $\psi : X \rightarrow 2 \mathbb{Z}$ est une fonction g\'en\'eriquement Nash-constructible en dimension $k$, dont le support modulo $2^2$ repr\'esente l'une des deux parties, images l'une de l'autre par $\sigma$, que l'on a obtenues par d\'ecoupage de la cha\^ine repr\'esent\'ee par l'adh\'erence de Zariski du support de $\pi^* c$.

On applique alors $\pi_*$ \`a $\gamma$:
\begin{eqnarray*}
\ \pi_* \gamma & = & \{\overline{x} \in X/G~|~\pi_*\left(\psi \times \pi^*(\varphi)\right) (\overline{x}) \notin 2^{k+ \alpha+2} \mathbb{Z} \} \\
& = & \{\overline{x} \in X/G~|~  \varphi(\overline{x}) \left(\psi(x) + \psi(\sigma(x))\right) \notin 2^{k+ \alpha+2} \mathbb{Z} \}.
\end{eqnarray*} 
Si $\overline{x} \in X/G$ v\'erifie $\varphi(\overline{x}) \left(\psi(x) + \psi(\sigma(x))\right) \notin 2^{k+ \alpha+2} \mathbb{Z}$, alors $\varphi(\overline{x}) \notin 2^{k + \alpha + 1} \mathbb{Z}$. R\'eciproquement, si $\overline{x} \in X/G$ v\'erifie $\varphi(\overline{x}) \notin 2^{k + \alpha + 1} \mathbb{Z}$, alors $x$ appartient au support de $\pi^*c$ (vrai au moins g\'en\'eriquement en dimension $k$) et donc soit au support de $\gamma$ soit au support de $\sigma~\gamma$, mais pas aux deux, et ainsi $\psi(x) + \psi(\sigma(x)) \notin 2^2 \mathbb{Z}$ et donc $\varphi(\overline{x}) \left(\psi(x) + \psi(\sigma(x))\right) \notin~2^{k+ \alpha+2} \mathbb{Z}$ (g\'en\'eriquement).

On obtient alors
$$\pi_* \gamma = \{ \overline{x} \in X/G~|~\varphi(\overline{x}) \notin 2^{k + \alpha + 1} \} = c$$
et $c$ appartient donc \`a l'image de $T^{\alpha+1}_k(X)$ par $\pi_*$.
\\

R\'eciproquement, soit $c \in T^{\alpha +1}_k(X)$. On a $c = \Supp~\psi~\modulo~2^{k + \alpha + 2}$ avec $\psi : X \rightarrow 2^{k+\alpha +1} \mathbb{Z}$ g\'en\'eriquement Nash-constructible en dimension $k$, et $(1 + \sigma) c = \Supp~\varphi~\modulo~2^{k + \alpha + 1}$ avec $\varphi$ g\'en\'eriquement Nash-constructible en dimension $k$. Comme $(1 + \sigma) c = \Supp~\frac{\psi + \psi \circ \sigma}{2}~\modulo~2^{k+ \alpha +1}$, on a alors 
$$\varphi \equiv \frac{\psi + \psi \circ \sigma}{2} ~\modulo~2^{k+ \alpha +1}$$
sur $X$.

On consid\`ere maintenant 
$$\pi_* c = \{\overline{x} \in X/G~|~\pi_*(\psi)(\overline{x}) \notin 2^{k + \alpha +2} \} = \{\overline{x} \in X/G~|~\psi(x) + \psi(\sigma(x)) \notin 2^{k + \alpha +2} \}.$$
On montre que l'on peut trouver une fonction g\'en\'eriquement Nash-constructible sur $X/G$, divisible par $2^{k+\alpha}$, qui repr\'esente $\pi_* c$. Pour cela, on applique le crit\`ere de l'\'eventail \cite{MCP} Theorem 4.10 \`a chacune des composantes irr\'eductibles de chaque carte affine de l'adh\'erence de Zariski $Y$ de $X/G$ et \`a la fonction 
$$f : y \mapsto \begin{cases}
\frac{\psi(x) + \psi(\sigma(x))}{2} \mbox{ si $y = \overline{x} \in X/G$}, \\
0 \mbox{ sinon, }
\end{cases}$$ 
d\'efinie et constructible sur $Y$.

Soit donc $F$ un \'eventail centr\'e en un point $y_0$ de $Y$, de cardinal $|F| \leq 2^{k + \alpha + 1}$. Comme $X$ est compact et l'action de $G$ sur $X$ \'etant libre, si le point $y_0 = \overline{x_0}$ est dans $X/G$, on peut relever $F$ en un \'eventail double $F' \sqcup \sigma(F')$ sur $X$ ($F'$ \'etant centr\'e en $x_0$ et $\sigma(F')$ en $\sigma(x_0)$), et sinon l'\'evaluation de $f$ en tout point de l'\'eventail sera nulle. Concentrons-nous donc sur le premier cas. On a alors 
$$\sum_{s \in F} f(s) = \sum_{s' \in F'} \frac{\psi\left(s'\right) + \psi\left(\sigma\left(s'\right)\right)}{2},$$
et cette somme, modulo $|F| = |F'| \leq 2^{k + \alpha +1}$, est \'egale \`a $\sum_{\sigma' \in F'} \varphi(\sigma')$ qui est elle-m\^eme \'egale \`a $0$ modulo $|F'| \leq 2^{k + \alpha +1}$, car $\varphi$ est g\'en\'eriquement Nash-constructible sur $X$ (\cite{MCP} Theorem 4.9 d'apr\`es \cite{Bon}).

Ainsi (\cite{MCP} Theorem 4.10), il existe une fonction $f_0 : Y \rightarrow \mathbb{Z}$, g\'en\'eriquement Nash-construc-tible, telle que, pour tout $y \in Y$, $f(y) - f_0(y) \equiv 0~ \modulo~ 2^{k + \alpha +1}$. En particulier, $f_0$ est divisible par $2^{k+\alpha}$ et 
$$\pi_* c = \Supp~f_0~\modulo~2^{k+\alpha+1}.$$
On peut de plus supposer $f_0$ g\'en\'eriquement constructible en dimension $k$, en rempla\c{c}ant, dans le raisonnement pr\'ec\'edent, $X$ par l'adh\'erence de Zariski du support de $c$. On a donc $\pi_* c \in \mathcal{N}_{\alpha} C_k(X/G)$.

\end{proof}

\begin{rem}

\begin{itemize}
	\item Les isomorphismes $\left(\mathcal{N}_{\alpha} C_k(X)\right)^G \cong \mathcal{N}_{\alpha} C_k\left(X/G\right)$ induisent des isomorphismes de complexes
$$\left(\mathcal{N}_{\alpha} C_*(X)\right)^G \cong \mathcal{N}_{\alpha} C_*\left(X/G\right),$$
induits par l'application quotient $\pi : X \rightarrow X/G$ par fonctorialit\'e de $\mathcal{N}_{\alpha} C_*$.
	\item Pour une $G$-vari\'et\'e alg\'ebrique r\'eelle $X$ quelconque munie d'une action de $G$ quelconque, en raisonnant de fa\c{c}on similaire, on obtient \'egalement les isomorphismes (de complexes)
	$$(1 + \sigma) \mathcal{N}_{\alpha} C_*(X) \cong im~\left( \mathcal{N}_{\alpha} C_*(X) \rightarrow C_*(X \setminus X^G / G) \right)$$ 
(o\`u $X \setminus X^G / G$ d\'esigne l'ensemble des points r\'eels de $X \setminus X^G$ quotient\'e par l'action restreinte de $G$).
\end{itemize}

\end{rem}

\end{document}